\documentclass[11pt,a4paper]{amsart}
\usepackage{graphicx}
\usepackage{amssymb}
\usepackage{amsmath}
\usepackage{amsthm}
\usepackage{amscd}
\usepackage[all,2cell]{xy}
\UseAllTwocells \SilentMatrices
\newtheorem{thm}{Theorem}[section]

\newtheorem{lem}[thm]{Lemma}

\theoremstyle{definition}

\theoremstyle{remark}
\newtheorem{rem}[thm]{\bf Remark}
\newtheorem{exm}[thm]{\bf Example}
\numberwithin{equation}{section}

\begin{document}
\title[Three results on Frobenius categories]
{Three results on Frobenius categories}
\author[Xiao-Wu Chen] {Xiao-Wu Chen}
\thanks{}
\date{\today}%

\thanks{This project was supported by Alexander von Humboldt
Stiftung and National Natural Science Foundation of China
(No.10971206).}
\thanks{E-mail:
xwchen$\symbol{64}$mail.ustc.edu.cn}
 \maketitle

\dedicatory{}%
\commby{}%

\begin{abstract}
This paper consists of three results on Frobenius categories: (1) we
give sufficient conditions on when a factor category of a Frobenius
category is still a Frobenius category; (2) we show that any
Frobenius category is equivalent to an extension-closed exact
subcategory of the Frobenius category formed by Cohen-Macaulay
modules over some additive category; this is an analogue of
Gabriel-Quillen's embedding theorem for Frobenius categories; (3) we
show that under certain conditions an exact category with enough
projective and enough injective objects allows a natural new exact
structure, with which the given category becomes Frobenius. Several applications
of the results are discussed.
\end{abstract}

\section{Introduction}

Recently Ringel and Schmidmeier study intensively the classification
problem in the (graded) submodule category over the truncated
polynomial algebra $k[t]/{(t^p)}$ (\cite{RS08}). Here, $k$ is a
field and $p\geq 1$ is a natural number. This problem goes back to
Birkhoff and is studied by Arnold and Simson. For an account of
the history, we refer to \cite{RS08}. The complexity of this
classification problem depends on the parameter $p$. According to
$p<6$, $p=6$ and $p>6$, the classification problem turns out to be
finite, tame and wild, respectively. We denote by
$\mathcal{S}(\widetilde{p})$ the graded submodule category which is
called  the category of Ringel-Schmidmeier in \cite{Ch09}. It has a
natural exact structure and becomes an exact category in the sense
of Quillen; moreover, it is a Frobenius category.

In  more recent work (\cite{KLM2}), Kussin, Lenzing and Meltzer give
a surprising link between the category $\mathcal{S}(\widetilde{p})$
of Ringel-Schmidmeier and the category of vector bundles on the
weighted projective line of type $(2, 3, p)$. To be more precise,
let $\mathbb{X}$ be the weighted projective line of type $(2, 3, p)$
in the sense of Geigle and Lenzing (\cite{GL}). Denote by ${\rm
vect}\; \mathbb{X}$ the category of vector bundles on $\mathbb{X}$.
It has a natural exact structure such that it is a Frobenius
category (\cite{KLM1}).  Following \cite{KLM2} we denote by
$\mathcal{F}$ the additive closure of the so-called fading line
bundles. Consider the factor category ${\rm vect}\;
\mathbb{X}/{[\mathcal{F}]}$ of ${\rm vect}\; \mathbb{X}$ modulo
those morphisms factoring through $\mathcal{F}$. One of the main
results in \cite{KLM2} states that there is an equivalence of
categories between ${\rm vect}\; \mathbb{X}/{[\mathcal{F}]}$ and
$\mathcal{S}(\widetilde{p})$; also see Example
\ref{exm:KLMcontinued}. From this equivalence the authors recover
some major results in \cite{RS08} via certain known results on
vector bundles over the weighted projective lines. Note that
according to $p< 6$, $p=6$ and $p>6$ the weighted projective line
$\mathbb{X}$ is domestic, tubular and wild, respectively.

We have noted above that the category $\mathcal{S}(\widetilde{p})$
of Ringel-Schmidmeier is Frobenius. Hence via the equivalence
mentioned above one infers that the  factor category ${\rm vect}\;
\mathbb{X}/{[\mathcal{F}]}$ is also a Frobenius category; see
\cite[Theorem A]{KLM2}. However a direct argument of this surprising
fact seems missing. More generally, one may ask when a factor
category of a Frobenius category is still Frobenius. This is one of
the motivations of the present paper. Another motivation is to understand
the minimal monomorphism  in the sense of Ringel-Schmidmeier (\cite{RS06}), 
which plays an important role in the study of  Auslander-Reiten sequences in submodule categories.

The paper  is organized as follows, which mainly consists of
three results on Frobenius categories.  We collect in Section 2 some basic facts and
notions on exact categories and Frobenius categories. In Section 3
we give sufficient conditions on when a factor category of a
Frobenius category is still Frobenius; see Theorem \ref{thm:partI}.
We apply the obtained result to recover \cite[Theorem A]{KLM2},
modulo a certain technical fact which is somehow hidden in
\cite{KLM2}. We also apply the result to the Frobenius category of
matrix factorizations. In Section 4 we prove a general result on
Frobenius categories: each Frobenius category is equivalent, as
exact categories, to an extension-closed exact subcategory of the
Frobenius category formed by Cohen-Macaulay modules over some
additive category; see Theorem \ref{thm:partII}. This can  be viewed
as an analogue of Gabriel-Quillen's embedding theorem for Frobenius
categories. We observe that the category
$\mathcal{S}(\widetilde{p})$ of Ringel-Schmidmeier can be viewed as
the category of Cohen-Macaulay modules over some graded algebras
(\cite{Ch09}). Together with this observation and the result in
Section 3, our general result recovers a part of \cite[Theorem
C]{KLM2}. In Section 5 we give sufficient conditions such that on an
exact category with enough projective and enough injective objects
there exists another natural exact structure, with which the given
category becomes Frobenius; see Theorem \ref{thm:partIII}.  An
application of this result allows us to interpret the minimal
monomorphism operation in \cite{RS06} as a triangle functor, which
is right adjoint to an inclusion triangle functor.

\section{Preliminaries on exact categories}
In this section we collect some basic facts and notions on exact
categories and Frobenius categories. The basic reference
is \cite[Appendix A]{Ke3}.  For a systematical treatment of
exact category, we refer to \cite{Bu10}.

\vskip 5pt

Let $\mathcal{A}$ be an additive category. A \emph{composable pair}
of morphisms is a sequence $X\stackrel{i} \rightarrow Y \stackrel{d}
\rightarrow Z$; such a composable pair is denoted by $(i, d)$. Two
composable pairs $(i, d)$ and $(i', d')$ are \emph{isomorphic}
provided that there are isomorphisms $f\colon X\rightarrow X'$,
$g\colon Y\rightarrow Y'$ and $h\colon Z\rightarrow Z'$ such that
$g\circ i=i'\circ f$ and $h\circ d=d' \circ g$. A composable pair
$(i, d)$ is called a \emph{kernel-cokernel pair}  provided that
$i={\rm Ker}\; d$ and $d={\rm Cok}\; i$.

 An \emph{exact structure} on an additive category $\mathcal{A}$
 is a chosen class $\mathcal{E}$ of  kernel-cokernel pairs in
 $\mathcal{A}$, which is closed under isomorphisms and
is subject to the following axioms (Ex0), (Ex1), (Ex1)$^{\rm op}$,
(Ex2) and (Ex2)$^{\rm op}$. A pair $(i, d)$ in the chosen class
$\mathcal{E}$ is called a \emph{conflation}, while $i$ is called an
\emph{inflation} and $d$ is called a \emph{deflation}. The pair
$(\mathcal{A}, \mathcal{E})$ is called an \emph{exact category} in
the sense of Quillen (\cite{Qui73}); sometimes we suppress the class
$\mathcal{E}$ and just say that $\mathcal{A}$ is an exact category.

Following \cite[Appendix A]{Ke3}, the axioms of  exact category are
listed as follows:

\begin{enumerate}
\item[(Ex0) \;  ] the identity morphism of the zero object is a deflation;

\item[(Ex1) \;  ] a composition of two deflations is a deflation;

\item[(Ex1)$^{\rm op}$] a composition of two inflations is an
inflation;

\item[(Ex2) \;  ] for a deflation $d\colon Y \rightarrow Z$ and a
morphism $f\colon Z'\rightarrow Z$ there exists a pullback diagram
such that $d'$ is a deflation:
\[\xymatrix{  Y' \ar@{.>}[r]^{d'} \ar@{.>}[d]^-{f'} & Z' \ar[d]^{f} \\
   Y  \ar[r]^-{d} & Z}\]

\item[(Ex2)$^{\rm op}$] for an inflation $i\colon X \rightarrow Y$ and a
morphism $f\colon X\rightarrow X'$ there exists a  pushout diagram
such that $i'$ is an inflation:
\[\xymatrix{  X \ar[r]^-{i} \ar[d]^-{f} & Y \ar@{.>}[d]^-{f'} \\
   X'  \ar@{.>}[r]^-{i'} & Y'}\]
\end{enumerate}

Let us remark that the axiom (Ex1)$^{\rm op}$ can be deduced from
the other axioms; see \cite[Appendix A]{Ke3}.

 For an exact category $\mathcal{A}$, a full additive subcategory
$\mathcal{B}\subseteq \mathcal{A}$ is said to be
\emph{extension-closed} provided that for any conflation
$X\stackrel{i} \rightarrow Y \stackrel{d}\rightarrow Z$ with $X,
Z\in \mathcal{B}$ we have $Y\in \mathcal{B}$. In this case, the
subcategory $\mathcal{B}$ inherits the exact structure from
$\mathcal{A}$ to become an exact category. We will call such a
subcategory an \emph{extension-closed exact subcategory}. Observe
that any abelian category  has a natural exact structure such that
conflations are induced by short exact sequences. Consequently, any
full additive subcategory in an abelian category which is closed
under extensions has a natural exact structure and then becomes an
exact category.

Recall that an additive functor $F\colon \mathcal{B}\rightarrow
\mathcal{A}$ between two exact categories is called \emph{exact}
provided that it sends conflations to conflations; an exact functor
$F\colon \mathcal{B}\rightarrow \mathcal{A}$ is said to be an
\emph{equivalence of exact categories} provided that $F$ is an
equivalence and there exists a quasi-inverse of $F$ which is exact.

From now on $\mathcal{A}$ is an exact category. We will need the
following two facts. For the first fact, we refer to the first step
in the proof of \cite[Proposition A.1]{Ke3}; for the second one, we
refer to the axiom c) in the proof of \cite[Proposition A.1]{Ke3}

\begin{lem}\label{lem:lem1}
Consider the diagram in {\rm (Ex2)}. Then the sequence
$$Y'\stackrel{\binom{d'}{-f'}}\longrightarrow Z'\oplus Y \stackrel{(f,
d)}\longrightarrow Z$$ is a conflation and we have a commutative
diagram such that the  two rows are conflations:
\[\xymatrix{  X\ar@{=}[d] \ar[r]^{i'} & Y' \ar[r]^{d'} \ar[d]^-{f'} & Z' \ar[d]^{f} \\
  X\ar[r]^-{f'\circ i'} &  Y  \ar[r]^-{d} & Z}\]
\end{lem}

\begin{lem}\label{lem:lem2}
Let $d$ be a morphism such that $d\circ e$ is a deflation for some
morphism $e$. Assume further that $d$ has a kernel. Then $d$ is a
deflation. \hfill $\square$
\end{lem}

Recall that an object $P$ in $\mathcal{A}$ is \emph{projective}
provided that the functor ${\rm Hom}_\mathcal{A}(P, -)$ sends
conflations to short exact sequences; this is equivalent to that any
deflation ending at $P$ splits. The exact category $\mathcal{A}$ is
said to \emph{have enough projective objects} provided that each
object $X$ fits into a deflation $d\colon P\rightarrow X$ with $P$
projective. Dually one has the notions of \emph{injective object}
and \emph{having enough injective objects}.

An exact category $\mathcal{A}$ is said to be \emph{Frobenius}
provided that it has enough projective and enough injective objects,
 and the class of projective objects coincides with the class of
injective objects (\cite[Section 3]{He60}). The importance of
Frobenius categories lies in  that they give rise naturally to
triangulated categories; see \cite{Ha1} and \cite[1.2]{Ke3}.

The following notion will be convenient for us: for a Frobenius
category $\mathcal{A}$, an extension-closed exact subcategory
$\mathcal{B}\subseteq \mathcal{A}$ is said to be \emph{admissible}
provided that each object $B$ in $\mathcal{B}$ fits into conflations
$B\rightarrow P\rightarrow B'$ and $B''\rightarrow Q\rightarrow B$
in $\mathcal{B}$ such that $P, Q$ are projective in $\mathcal{A}$.
Note that an admissible subcategory $\mathcal{B}$ of a Frobenius
category $\mathcal{A}$ is still Frobenius; moreover, an object $B$
in $\mathcal{B}$ is projective if and only if it is projective
viewed as an object in $\mathcal{A}$.

\section{ Factor category of Frobenius category}

In this section we study a certain factor category of a Frobenius
category. We give sufficient conditions on when the factor category
inherits the exact structure from  the given Frobenius category such
that it becomes a Frobenius category. As an application, our result
specializes to \cite[Theorem A]{KLM2} modulo certain technical
results which are somehow hidden in \cite{KLM2}. We give an example to apply
our result to the category of matrix factorizations (\cite{Ei80}).

\vskip 5pt

Let $(\mathcal{A}, \mathcal{E})$ be a Frobenius category. Denote by
$\mathcal{P}$ the full subcategory consisting of projective objects.
Let $\mathcal{F}\subseteq \mathcal{P}$ be a full additive
subcategory. For two objects $X, Y$ in $\mathcal{A}$ denote by
$[\mathcal{F}](X, Y)$ the subgroup of ${\rm Hom}_\mathcal{A}(X, Y)$
consisting of those morphisms which factor through an object in
$\mathcal{F}$. Denote by $\mathcal{A}/[\mathcal{F}]$ the
\emph{factor category} of $\mathcal{A}$ modulo $\mathcal{F}$: the
objects are the same as the ones in $\mathcal{A}$, for two objects
$X$ and $Y$ the Hom space is given by the quotient group ${\rm Hom}_\mathcal{A}(X,
Y)/[\mathcal{F}](X, Y)$ and the composition is induced by the one in
$\mathcal{A}$; compare \cite[p.101]{ARS}. Note that the factor
category $\mathcal{A}/[\mathcal{F}]$ is an additive category.

Denote by $\pi_\mathcal{F}\colon \mathcal{A}\rightarrow
\mathcal{A}/[\mathcal{F}]$ the canonical functor. Denote by
$\mathcal{E}_\mathcal{F}$ the class of composable pairs in
$\mathcal{A}/[\mathcal{F}]$  which are isomorphic to composable
pairs $(\pi_\mathcal{F}(i), \pi_\mathcal{F}(d))$ for $(i, d)\in
\mathcal{E}$.

The case $\mathcal{F}=\mathcal{P}$ is of particular interest, since
the corresponding factor category, known as the \emph{stable
category} of $\mathcal{A}$ and  denoted by
$\underline{\mathcal{A}}$, has a natural triangulated structure. In
this case the canonical functor $\pi_\mathcal{P}\colon
\mathcal{A}\rightarrow \underline{\mathcal{A}}$ sends conflations to
exact triangles. For details, see \cite[Chapter I, Section 2]{Ha1}.

We are interested in the following question: when the factor
category $\mathcal{A}/[\mathcal{F}]$ becomes a Frobenius category
such that its exact structure is given by $\mathcal{E}_\mathcal{F}$?
Note that in general the case $\mathcal{F}=\mathcal{P}$ will not
meet the requirement. The aim of this section is to give a partial
answer to this question.

\vskip 5pt

 Recall that a \emph{pseudo-kernel} of a morphism $f\colon
X\rightarrow Y$ is a morphism $c\colon Y\rightarrow C$ such that
$c\circ f=0$ and it satisfies that any morphism $c'\colon
Y\rightarrow C'$ with $c'\circ f=0$ factors through $c$. Dually one
has the notion of \emph{pseudo-cokernel}; see \cite[Section
2]{AS81}.

Recall that for a subcategory $\mathcal{S}$ of $\mathcal{A}$,  a
morphism $f\colon S \rightarrow X$ is said to be a \emph{right
$\mathcal{S}$-approximation} of $X$ provided that $S\in \mathcal{S}$
and any morphism from an object in $\mathcal{S}$ to $X$ factors
through $f$. Dually one has the notion of \emph{left
$\mathcal{S}$-approximation}; see \cite[Section 1]{AR91}.

\vskip 5pt

Our first result is as follows, which gives sufficient conditions on
when the pair $(\mathcal{A}/[\mathcal{F}], \mathcal{E}_\mathcal{F})$
is a Frobenius category.

\begin{thm}\label{thm:partI}
Let $(\mathcal{A}, \mathcal{E})$ be a Frobenius category and let
$\mathcal{P}$ denote the subcategory of projective objects. Suppose
that $\mathcal{F}\subseteq \mathcal{P}$ satisfies the following
conditions:
\begin{enumerate}
\item any object $A$ in $\mathcal{A}$ fits into a sequence
$$A\stackrel{i_A}\longrightarrow F_A\stackrel{p_A}\longrightarrow
P_A$$ such that $i_A$ is a left $\mathcal{F}$-approximation of $A$,
$P_A\in \mathcal{P}$ and $p_A$ is a pseudo-cokernel of $i_A$;

\item any object $A$ in $\mathcal{A}$ fits into a sequence
$$P^A\stackrel{i^A}\longrightarrow F^A\stackrel{p^A}\longrightarrow
A$$
such that $p^A$ is a right $\mathcal{F}$-approximation of $A$,
$P^A\in \mathcal{P}$ and $i^A$ is a pseudo-kernel of $p^A$.
\end{enumerate}
Then the pair $(\mathcal{A}/[\mathcal{F}], \mathcal{E}_\mathcal{F})$
is a Frobenius category.
\end{thm}

\begin{proof} In the proof, we write $\pi_\mathcal{F}$ as $\pi$. We
will divide the proof into three steps.

\vskip 3pt

 \emph{Step 1.} We will  first show that the composable pairs
in $\mathcal{E}_\mathcal{F}$ are kernel-cokernel pairs. It suffices
to show that for any conflation $X\stackrel{i}\rightarrow
Y\stackrel{d}\rightarrow Z$ in $\mathcal{A}$ we have $\pi(i)={\rm
Ker}\; \pi(d)$ and $\pi(d)={\rm Cok}\; \pi(i)$. We will only show
that $\pi(i)={\rm Ker}\; \pi(d)$, and the remaining equality is shown by a
dual argument.

To show that $\pi(i)$ is mono, it suffices to show that any morphism
$a\colon A \rightarrow X$ in $\mathcal{A}$ having the property
$i\circ a\in [\mathcal{F}](A, Y)$ necessarily lies in
$[\mathcal{F}](A, X)$. Consider the sequence in (1) for $A$. Since
$i\circ a\colon A\rightarrow Y$ factors through an object in
$\mathcal{F}$ and $i_A\colon A\rightarrow F_A$ is a left
$\mathcal{F}$-approximation, there is a morphism $t\colon
F_A\rightarrow Y$ such that $i\circ a=t\circ i_A$. Using that $p_A$
is a pseudo-cokernel of $i_A$, we have a morphism $s\colon
P_A\rightarrow Z$ making the following diagram commute
\[\xymatrix{ A \ar[d]^-a \ar[r]^-{i_A} & F_A \ar@{.>}[d]^{t} \ar[r]^-{p_A} & P_A
\ar@{.>}[d]^-{s}\\
X \ar[r]^-i & Y \ar[r]^-d & Z}\] Since $P_A$ is projective and $(i,
d)$ is a conflation, we may lift $s$ to a morphism $s'\colon
P_A\rightarrow Y$ such that $d\circ s'=s$. Then we have
$$d\circ
(t-s'\circ p_A)=d\circ t-s\circ p_A=0.$$ Since $i={\rm Ker}\;d $,
there exists $a'\colon F_A\rightarrow X$ such that $i\circ
a'=t-s'\circ p_A$. Composing the two sides with $i_A$, we get
$i\circ a'\circ i_A=t\circ i_A=i\circ a$. Note that $i$ is mono and
$F_A\in \mathcal{F}$. Then we have $a=a'\circ i_A$ and it lies in
$[\mathcal{F}](A, X)$.

Having shown that $\pi(i)$ is mono, it suffices to show that
$\pi(i)$ is a pseudo-kernel of $\pi(d)$. Then we have $\pi(i)={\rm
Ker}\; \pi(d)$. For this end, take a morphism $a\colon A\rightarrow
Y$ such that $d\circ a\in [\mathcal{F}](A, Z)$. We will show that
$\pi(a)$ factors through $\pi(i)$. Assume that $d\circ a$ factors as
$A\stackrel{x}\rightarrow F\stackrel{y}\rightarrow Z$ with $F\in
\mathcal{F}$. Since $F$ is projective and $(i, d)$ is a conflation,
we may lift $y$ to a morphism $y'\colon F\rightarrow Y$ such that
$d\circ y'=y$. Then we have
 $$d\circ (a-y'\circ x)=d\circ a-y\circ
x=0.$$
Hence there exists a morphism $a'\colon A\rightarrow X$ such
that $a-y'\circ x=i\circ a'$. Note that $F\in \mathcal{F}$. Applying
$\pi$ we get $\pi(a)=\pi(i)\circ \pi(a')$.

\vskip 3pt

 \emph{Step 2.} We will show next that the pair
$(\mathcal{A}/[\mathcal{F}], \mathcal{E}_\mathcal{F})$ is an exact
category. Note that by definition a morphism $\delta \colon
\pi(Y)\rightarrow \pi(Z)$ is a deflation if and only if there exist
morphisms $a\colon Y\rightarrow Y'$ and $b\colon Z'\rightarrow Z$
such that $\pi(a)$ and $\pi(b)$ are isomorphisms, and a deflation
$d\colon Y'\rightarrow Z'$ in $\mathcal{A}$ such that we have a
\emph{factorization} $\delta=\pi(b)\circ \pi(d)\circ \pi(a)$. The axiom
(Ex0) is trivial.

To show (Ex1), assume that we are given two deflations $\delta\colon
\pi(Y)\rightarrow \pi(Z)$ and $\gamma\colon \pi(Z)\rightarrow
\pi(W)$ in $\mathcal{A}/[\mathcal{F}]$. We may assume that $\delta$
and $\gamma $ factor as $\pi(Y)\stackrel{\pi(a)}\rightarrow
\pi(Y')\stackrel{\pi(d)}\rightarrow
\pi(Z')\stackrel{\pi(b)}\rightarrow \pi(Z)$
 and $\pi(Z)\stackrel{\pi(x)}\rightarrow \pi(Z'')\stackrel{\pi(e)}\rightarrow \pi(W')
 \stackrel{\pi(y)}\rightarrow \pi(W)$, respectively. Here $d\colon Y'\rightarrow Z'$ and
 $e\colon Z''\rightarrow W'$ are deflations in $\mathcal{A}$. Take a
 morphism $z\colon Z''\rightarrow Z$ such that $\pi(z)=(\pi(x)\circ
 \pi(b))^{-1}$. By (Ex2) we have  the pullback diagram in $\mathcal{A}$
 \[\xymatrix{
Y'' \ar@{.>}[r]^-{d'} \ar@{.>}[d]^-{z'} & Z'' \ar[d]^{z}\\
Y'\ar[r]^-{d} & Z'
 }\]
 such that $d'$ is a deflation. By Lemma \ref{lem:lem1} the sequence
$Y''\stackrel{\binom{d'}{-z'}}\rightarrow Z''\oplus Y'\stackrel{(z,
d)}\rightarrow
 Z'$ is a conflation in $\mathcal{A}$. By the first step, applying $\pi$ to this sequence we
 get a kernel-cokernel pair in the factor category $\mathcal{A}/[\mathcal{F}]$.
 In particular, the diagram above is
 still a pullback diagram in $\mathcal{A}/[\mathcal{F}]$. Hence the
 fact that $\pi(z)$ is an isomorphism implies that $\pi(z')$ is also an
 isomorphism. Take a morphism $a'\colon Y'\rightarrow Y''$  in $\mathcal{A}$
 such that $\pi(a')=\pi(z')^{-1}$.
 Then $\gamma \circ \delta$ factors as $$\pi(Y)\stackrel{\pi(a'\circ a)}\longrightarrow \pi(Y'')
 \stackrel{\pi(e\circ d')}\longrightarrow \pi(W')\stackrel{\pi(y)}\longrightarrow
 \pi(W).$$
By (Ex1) $e \circ d'$ is a deflation in $\mathcal{A}$. Observe that
both $\pi(a'\circ a)$ and $\pi(y)$ are isomorphisms. Then we have
that $\gamma\circ \delta$ is a deflation in
$\mathcal{A}/[\mathcal{F}]$, proving the axiom (Ex1). Dually one
shows (Ex1)$^{\rm op}$.

To show (Ex2), take a deflation $\delta\colon \pi(Y)\rightarrow
\pi(Z)$ and a morphism $\pi(f)\colon \pi(Z')\rightarrow \pi(Z)$.
Without loss of generality we may assume that $\delta=\pi(d)$ for a
deflation $d\colon Y\rightarrow Z$ in $\mathcal{A}$. Then we apply
(Ex2) for $\mathcal{A}$ to get a pullback diagram in $\mathcal{A}$.
As above, using Lemma \ref{lem:lem1} and the first step, the
obtained diagram is also a pullback diagram in the factor category
$\mathcal{A}/[\mathcal{F}]$. This proves the axiom (Ex2) for
$\mathcal{A}/[\mathcal{F}]$. Dually one shows (Ex2)$^{\rm op}$.

\vskip 3pt
\emph{Step 3.} The exact category $(\mathcal{A}/[\mathcal{F}],
\mathcal{E}_\mathcal{F})$ is Frobenius. Recall that up to
isomorphism conflations in $\mathcal{A}/[\mathcal{F}]$ are given by the
images of  conflations in $\mathcal{A}$. Then it follows immediately
that objects in $\mathcal{P}/[\mathcal{F}]$ are projective and
injective in the exact category $(\mathcal{A}/[\mathcal{F}],
\mathcal{E}_\mathcal{F})$; moreover, each object $\pi(X)$ in
$\mathcal{A}/[\mathcal{F}]$ admits a deflation $\pi(P)\rightarrow
\pi(X)$ and an inflation $\pi(X)\rightarrow \pi(I)$ with $\pi(P),
\pi(I)\in \mathcal{P}/[\mathcal{F}]$. From these, one concludes
immediately that the exact category $(\mathcal{A}/[\mathcal{F}],
\mathcal{E}_\mathcal{F})$ is Frobenius.
\end{proof}

\begin{rem}
As shown in the third step above, the full subcategory of
$\mathcal{A}/[\mathcal{F}]$ consisting of projective objects is
equal to $\mathcal{P}/[\mathcal{F}]$. Using again the fact that up
to isomorphism conflations in $\mathcal{A}/[\mathcal{F}]$ are given
by the images of  conflations in $\mathcal{A}$, we have an
identification
$\underline{\mathcal{A}}=\underline{\mathcal{A}/[\mathcal{F}]}$ of
triangulated categories. \hfill $\square$
\end{rem}

We will apply Theorem \ref{thm:partI} in two examples.
We begin with our motivating example. We will see that, modulo certain technical results in
\cite{KLM2}, Theorem \ref{thm:partI} specializes to \cite[Theorem
A]{KLM2}.

\begin{exm}\label{exm:KLM}
Let $k$ be a field and $p\geq 2$ be a natural number. Let
$\mathbb{X}$ be the \emph{weighted projective line} of type $(2, 3,
p)$ in the sense of Geigle and Lenzing (\cite{GL}). Denote by ${\rm
coh}\; \mathbb{X}$ the abelian category of coherent sheaves on
$\mathbb{X}$ and by $\mathcal{O}$ the structure sheaf on
$\mathbb{X}$. Denote by $L$ the rank one abelian group on three
generators $\vec{x}_1$, $\vec{x_2}$, $\vec{x}_3$ subject to the
relations $2\vec{x}_1=3\vec{x}_2=p\vec{x}_3$. Recall that the group
$L$ acts on ${\rm coh}\; \mathbb{X}$. We denote  the action of an
element $\vec{x}\in L$ on a sheaf $E$ by $E(\vec{x})$.

Denote by  ${\rm vect}\; \mathbb{X}$ the full subcategory of ${\rm
coh}\; \mathbb{X}$ consisting of vector bundles. Recall that all the
line bundles on $\mathbb{X}$ are given by $\mathcal{O}(\vec{x})$ for
$\vec{x}\in L$; moreover, $\mathcal{O}(\vec{x})\simeq
\mathcal{O}(\vec{y})$ implies that $\vec{x}=\vec{y}$. In other
words, the Picard group of $\mathbb{X}$ is isomorphic to $L$; see
\cite[Proposition 2.1]{GL}.  Recall that the subcategory ${\rm
vect}\; \mathbb{X}\subseteq {\rm coh}\; \mathbb{X}$ is closed under
extensions and then it has a natural exact structure. However with
this exact structure the category  ${\rm vect}\; \mathbb{X}$ is not
Frobenius.

Following \cite{KLM1}  a short exact sequence $\eta\colon
0\rightarrow E'\rightarrow E\rightarrow E''\rightarrow 0 $ of vector bundles is
\emph{distinguished} provided that the sequences ${\rm
Hom}(\mathcal{O}(\vec{x}), \eta)$ are exact for all $\vec{x}\in L$.
By Serre duality this is equivalent to that the sequences ${\rm
Hom}(\eta, \mathcal{O}(\vec{x}))$ are exact for all $\vec{x}\in L$.
Observe that the category ${\rm vect}\; \mathbb{X}$ of vector
bundles is an exact category such that conflations are induced by
distinguished short exact sequences; compare Lemma \ref{lem:lem3}.
We denote by $\mathcal{A}$ this exact category. Moreover, the exact
category $\mathcal{A}$ is Frobenius such that its subcategory
$\mathcal{P}$ of projective objects is equal to the additive closure
of all line bundles. For details, see \cite{KLM1}.

The following terminology is taken from \cite{KLM2}. A line bundle
$\mathcal{O}(\vec{x})$ is said to be \emph{fading} provided that
$\vec{x}\notin \mathbb{Z}\vec{x}_3\cup \vec{x}_2+
\mathbb{Z}\vec{x}_3$. Take $\mathcal{F}\subseteq \mathcal{P}$ to be
the additive closure of these fading line bundles. We claim that the
subcategory $\mathcal{F}$ satisfies the conditions in Theorem
\ref{thm:partI}. Then it follows from Theorem \ref{thm:partI} that
the factor category $\mathcal{A}/[\mathcal{F}]$ inherits the
Frobenius exact structure from the one of $\mathcal{A}$; this is
\cite[Theorem A]{KLM2}.

In fact, the proof of \cite[Proposition 3.13]{KLM2} yields the
following technical fact: for a vector bundle $E$ there is a short
exact sequence  $0\rightarrow  E\stackrel{\alpha}\rightarrow C
\rightarrow P_1\rightarrow 0$ with $C\in \mathcal{F}$ and $P_1\in
\mathcal{P}$; moreover, the morphism $\alpha$ is a left
$\mathcal{F}$-approximation (by \cite[Lemma 3.12 (2)]{KLM2}). Here
we are consistent in notation with the proof of \cite[Proposition
3.13]{KLM2}. Note that one  has a dual version of this result using
the duality $d\colon \mathcal{A}\rightarrow \mathcal{A}$ in the
proof of \cite[Proposition 3.2]{KLM2}. \hfill $\square$
\end{exm}

The second example shows that a certain factor category of the
category of matrix factorizations has a Frobenius exact structure.

\begin{exm}\label{exm:mf}
Let $R$ be a commutative noetherian ring and let $f\in R$ be a
regular element. Recall that a \emph{matrix factorization} of $f$ is
a composable pair $P^0\stackrel{d_P^0} \rightarrow P^1
\stackrel{d_P^1}\rightarrow P^0$ consisting of finitely generated
projective $R$-modules such that $d_P^1\circ d_P^0=f\; {\rm
Id}_{P^0}$ and $d_P^0\circ d_P^1=f\; {\rm Id}_{P^1}$; a morphism
$(f^0, f^1)\colon(d_P^0, d_P^1) \rightarrow (d_Q^0, d_Q^1) $ between
matrix factorizations consists of two morphisms $f^0\colon
P^0\rightarrow Q^0$ and $f^1\colon P^1\rightarrow Q^1$ of
$R$-modules such that $d_Q^0\circ f^0=f^1\circ d_P^0$ and
$d_Q^1\circ f^1=f^0\circ d_P^1$.  Observe that since $f$ is regular,
the two morphisms $d_P^0$ and $d_P^1$ in a matrix factorization are
mono. For details, see \cite[Section 5]{Ei80}.

Denote by ${\rm MF}_R(f)$ the category of matrix factorizations of
$f$. It has a natural exact structure such that a sequence
$(d_{P'}^0, d_{P'}^1)\rightarrow (d_{P}^0, d_{P}^1)\rightarrow
(d_{P''}^0, d_{P''}^1)$ is a conflation if and only if the
corresponding sequences $0\rightarrow P'^i\rightarrow P^i\rightarrow
P''^i\rightarrow 0$ of $R$-modules are short exact, $i=0, 1$.
Moreover, with this exact structure ${\rm MF}_R(f)$ is a Frobenius
category, and its projective objects are equal to direct summands of
an object of the form $({\rm Id}_P, f\; {\rm Id}_P)\oplus (f\; {\rm
Id}_P, {\rm Id}_P)$ for a projective $R$-module $P$; compare
\cite[Chapter I, 3.2]{Ha1} and \cite[Example 5.3]{Ke}.

Denote by $\mathcal{F}$ the full subcategory of ${\rm MF}_R(f)$ consisting of
objects of the form $({\rm Id}_P, f\; {\rm Id}_P)$ for a projective $R$-module $P$. We claim that
$\mathcal{F}$ satisfies the conditions in Theorem \ref{thm:partI}. Indeed, for a matrix
factorization $(d_P^0, d_P^1)$, the following two sequences
$$(d_P^0, d_P^1) \stackrel{(d_P^0, {\rm Id}_{P^1})}\longrightarrow ({\rm Id}_{P^1}, f\; {\rm Id}_{P^1})
\longrightarrow (0, 0)$$
and
$$(0, 0)\rightarrow ({\rm Id}_{P^0}, f\; {\rm Id}_{P^0}) \stackrel{({\rm Id}_{P^0}, d_P^0)}\longrightarrow (d_P^0, d_P^1)$$
are the required sequences in (1) and (2), respectively.  In this way, we get a factor Frobenius
category ${\rm MF}_R(f)/[\mathcal{F}]$.  \hfill $\square$
\end{exm}

\section{Frobenius category and Cohen-Macaulay module}

In this section we will show that any Frobenius category is
equivalent, as exact categories, to an admissible subcategory of the
Frobenius category formed by Cohen-Macaulay modules over an additive
category. This is an analogue of Gabriel-Quillen's embedding theorem
for Frobenius categories; see \cite{Bu10, Ke3}. In particular, our
result suggests that the category of Cohen-Macaulay modules serves
as a standard model for Frobenius categories. We apply the obtained
result to recover a part of \cite[Theorem C]{KLM2}. We also make an application to
the category of matrix factorizations.

\vskip 5pt

Let $\mathcal{C}$ be an additive category. Denote by ${\rm Mod}\;
\mathcal{C}$ the (large)  abelian category of additive contravariant
functors from $\mathcal{C}$ to the category of abelian groups; by
abuse of terminology these functors are called
$\mathcal{C}$-\emph{modules}. Note that exact sequences of
$\mathcal{C}$-modules are given by sequences of functors over
$\mathcal{C}$, which are exact taking values at each object $C\in
\mathcal{C}$.

For an object $C$ in $\mathcal{C}$, denote by $H_C={\rm
Hom}_\mathcal{C}(-, C)$ the corresponding representable functor.
This gives rise to the \emph{Yoneda functor} $H\colon
\mathcal{C}\rightarrow {\rm Mod}\; \mathcal{C}$. Yoneda Lemma says
that there exists a natural isomorphism ${\rm Hom}_{{\rm Mod}\;
\mathcal{C}}(H_C, M)\simeq M(C)$ for each object $C\in \mathcal{C}$
and $M\in {\rm Mod}\; \mathcal{C}$. From these one infers that the
Yoneda functor $H$ is fully faithful and the modules  $H_C$ are
projective for all $C\in \mathcal{C}$. Recall that a
$\mathcal{C}$-module $M$ is \emph{finitely generated} provided that
there exists an epimorphism $H_C\rightarrow M$ for some object $C\in
\mathcal{C}$. Observe that a $\mathcal{C}$-module is finitely
generated projective if and only if  it is a direct summand of $H_C$
for an object $C\in \mathcal{C}$. For details, we refer to
\cite{Mit72}.

Recall that a cochain complex $P^\bullet=(P^n, d^n\colon
P^n\rightarrow P^{n+1})_{n\in \mathbb{Z}}$ consisting of finitely
generated projective $\mathcal{C}$-modules is said to be
\emph{totally acyclic} provided that it is acyclic and for each
object $C$ the Hom complex ${\rm Hom}_{{\rm Mod}\;
\mathcal{C}}(P^\bullet, H_C)$ is acyclic; compare \cite[p.400]{AM}.
Following \cite{Buc87} and \cite{Bel3} a $\mathcal{C}$-module $M$ is
said to be (maximal) \emph{Cohen-Macaulay} provided that there exists a
totally acyclic complex $P^\bullet$ such that the $0$-th cocycle
$Z^0(P^\bullet)$ is isomorphic to $M$. In this case, the complex
$P^\bullet$ is said to be a \emph{complete resolution} of $M$.
Observe that a finitely generated projective $\mathcal{C}$-module
$P$ is Cohen-Macaulay, since we may take its complete resolution as
$\cdots \rightarrow 0\rightarrow P\stackrel{{\rm Id}_P}\rightarrow
P\rightarrow 0\rightarrow \cdots$. Note that in the literature,
Cohen-Macaulay modules are also called \emph{modules of G-dimension
zero} (\cite{ABr69}) and  \emph{Gorenstein-projective modules}
(\cite{EJ}). Let us remark that Cohen-Macaulay modules are closely related
to singularity categories (\cite{Buc87,Or04,Ch10}).

Denote by ${\rm CM}(\mathcal{C})$ the full subcategory of ${\rm
Mod}\; \mathcal{C}$ consisting of Cohen-Macaulay
$\mathcal{C}$-modules. Note that since each Cohen-Macaulay module is
finitely generated, the category ${\rm CM}(\mathcal{C})$  has small
Hom sets. Observe that ${\rm CM}(\mathcal{C})\subseteq {\rm Mod}\;
\mathcal{C}$ is closed under extensions; compare \cite[Propositon
5.1]{AR91}. Then it becomes an exact category such that conflations
are induced by short exact sequences with terms in ${\rm
CM}(\mathcal{C})$.

The following result is well known; compare \cite[Proposition
3.1(1)]{Ch10}. For the definition of an admissible subcategory, see
Section 2.

\begin{lem}\label{lem:lem3.5}
The exact category ${\rm CM}(\mathcal{C})$ is Frobenius; moreover,
its projective objects are equal to finitely generated projective
$\mathcal{C}$-modules. Consequently, any admissible subcategory of
${\rm CM}(\mathcal{C})$ is a Frobenius category.
\end{lem}

\begin{proof}
Observe first that for a Cohen-Macaulay $\mathcal{C}$-module $M$ and
a finitely generated projective $\mathcal{C}$-module $P$ we have
${\rm Ext}^i_{{\rm Mod}\; \mathcal{C}}(M, P)=0$ for $i\geq 1$;
compare \cite[Lemma 2.1]{CFH06}. Hence the object $P$ is injective
in ${\rm CM}(\mathcal{C})$; while it is clearly projective in ${\rm
CM}(\mathcal{C})$. Observe from the definition that for each
Cohen-Macaulay module $M$ with its complete resolution $P^\bullet$,
we have two conflations $Z^{-1}(P^\bullet)\rightarrow
P^{-1}\rightarrow M$ and $M\rightarrow P^0\rightarrow
Z^1(P^\bullet)$. These two conflations imply that the exact category
${\rm CM}(\mathcal{C})$ has enough projective and enough injective objects;
moreover, from these one infers that the class of projective objects
coincides with the class of injective objects, both of which are
equal to the class of finitely generated projective
$\mathcal{C}$-modules. This shows that the category ${\rm CM}(\mathcal{C})$ is a Frobenius
category, and the last statement follows immediately; see Section 2.
\end{proof}

Let $(\mathcal{A}, \mathcal{E})$ be a Frobenius category. Denote by
$\mathcal{P}$ the full subcategory of its projective objects.
Consider the category ${\rm Mod}\; \mathcal{P}$ of
$\mathcal{P}$-modules. For each object $A$ in $\mathcal{A}$ denote
by $h_A$ the $\mathcal{P}$-module obtained by restricting the
functor $H_A={\rm Hom}_\mathcal{A}(-, A)$ on $\mathcal{P}$. This
yields a functor $h\colon \mathcal{A}\rightarrow {\rm Mod}\;
\mathcal{P}$ sending $A$ to $h_A$; such a functor is known as the
\emph{restricted Yoneda functor}. Observe that for an object $P\in
\mathcal{P}$ we have $h_P=H_P$.

\vskip 5pt

Recall from Lemma \ref{lem:lem3.5} that an admissible subcategory of
the category of Cohen-Macaulay modules is Frobenius. In fact, all
Frobenius categories arise in this way. This is our second result,
which is an analogue of Gabriel-Quillen's embedding theorem for
Frobenius categories; see \cite[Proposition A.2]{Ke3} and
\cite[Theorem A.1]{Bu10}.

\begin{thm}\label{thm:partII}
Use the notation as above. Then the restricted Yoneda functor
$h\colon \mathcal{A}\rightarrow {\rm Mod}\; \mathcal{P}$ induces an
equivalence  of exact categories between $\mathcal{A}$
 and an admissible subcategory of ${\rm CM}(\mathcal{P})$.
 \end{thm}

\begin{proof}
We will divide the proof into four steps. First observe that the functor
$h$ sends conflations in $\mathcal{A}$ to short exact sequences of
$\mathcal{P}$-modules, and sends projective objects in $\mathcal{A}$ to
representable functors over $\mathcal{P}$, in particular, finitely
generated projective $\mathcal{P}$-modules.

\vskip 3pt

\emph{Step 1.} We will show that for each object $A\in \mathcal{A}$
the $\mathcal{P}$-module $h_A$ is Cohen-Macaulay. For this, take
conflations $\eta^i\colon A^i\rightarrow P^i \rightarrow A^{i+1}$
such that $A^0=A$ and $P^i$'s are projective for $i\in \mathbb{Z}$.
Applying $h$ to these conflations we get short exact sequences
$0\rightarrow h_{A^i}\rightarrow h_{P^i} \rightarrow
h_{A^{i+1}}\rightarrow 0$. Splicing these short exact sequences we
get an acyclic complex $h_{P^\bullet}$ of finitely generated
projective $\mathcal{P}$-modules which satisfies that
$Z^0(h_{P^\bullet})\simeq h_A$. It remains to show that the complex
$h_{P^\bullet}$ satisfies that for each object $P\in \mathcal{P}$
the Hom complex ${\rm Hom}_{{\rm Mod}\; \mathcal{P}}(h_{P^\bullet},
H_P)$ is acyclic. Here $H_P$ denotes the representable functor
corresponding to $P$. Using Yoneda Lemma this Hom complex is
isomorphic to the Hom complex ${\rm Hom}_\mathcal{A}(P^\bullet, P)$.
Here the complex $P^\bullet$ in $\mathcal{A}$ is constructed by
splicing  the conflations $\eta^i$ together. Then the Hom complex
${\rm Hom}_\mathcal{A}(P^\bullet, P)$ is acyclic, since it is
constructed  by splicing the short exact sequence ${\rm
Hom}_\mathcal{P}(\eta^i, P)$ together; here we use the fact that the
object $P$ is injective in $\mathcal{A}$. Consequently the complex
$h_{P^\bullet}$ is totally acyclic and then the $\mathcal{P}$-module
$h_A$ is Cohen-Macaulay.

\vskip 3pt
 \emph{Step 2.} We will show that the functor $h$ is fully faithful. This
is indeed fairly standard; compare the argument in
\cite[p.102]{ARS}. We will only show the fullness, and by a similar
argument one can show the faithfulness.

For an object $A\in \mathcal{A}$, we obtain from the conflations
$\eta^{-1}$ and $\eta^{-2}$ in the first step a cokernel sequence
$P^{-2}\rightarrow P^{-1}\rightarrow A\rightarrow 0$. This sequence
induces a projective presentation $H_{P^{-2}}\rightarrow
H_{P^{-1}}\rightarrow h_A \rightarrow 0$ of $\mathcal{P}$-modules.
Similarly for another object $A'$ we get a projective presentation
$H_{P'^{-2}}\rightarrow H_{P'^{-1}}\rightarrow h_{A'}\rightarrow 0$.
Given a morphism $\theta\colon h_A\rightarrow h_{A'}$, there exists
a commutative diagram
\[\xymatrix{
H_{P^{-2}}\ar@{.>}[d]^-{\theta^{-2}} \ar[r] & H_{P^{-1}}\ar@{.>}[d]^-{\theta^{-1}}
\ar[r] & h_A \ar[r] \ar[d]^-\theta & 0\\
 H_{P'^{-2}} \ar[r] & H_{P'^{-1}} \ar[r] & h_{A'} \ar[r] & 0
 }\]
Observe that by Yoneda Lemma  there exist morphisms $\mu^{-i}\colon
P^{-i}\rightarrow P'^{-i}$ such that $h_{\mu^{-i}}=\theta^{-i}$;
moreover, these two morphisms make the left side square in the
following diagram  commute.
\[\xymatrix{
P^{-2} \ar[d]^-{\mu^{-2}} \ar[r] & P^{-1} \ar[d]^-{\mu^{-1}}
\ar[r] & A \ar[r] \ar@{.>}[d]^-\mu & 0\\
 P'^{-2} \ar[r] & P'^{-1} \ar[r] & A \ar[r] & 0
 }\]
Since the two rows in the diagram above are cokernel sequences, one
infers that there exists $\mu\colon A \rightarrow A'$ making the
diagram commute. It is direct to see that $h_\mu=\theta$ and this
proves that the functor $h$ is full.

\vskip 3pt

 \emph{Step 3.} Denote by ${\rm Im}\; h$ the \emph{essential
image} of the functor $h$. We have shown that ${\rm Im}\; h\subseteq
{\rm CM}(\mathcal{P})$. We will now show that it is
extension-closed. Note that the functor $h$ sends projective objects
to projective modules, and sends conflations to short exact
sequences. This will imply that ${\rm Im}\; h$ is an admissible
subcategory of ${\rm CM}(\mathcal{P})$; see Section 2.

Take a conflation $h_X \rightarrow M \rightarrow h_Y$ in ${\rm
CM}(\mathcal{P})$ with $X, Y\in \mathcal{A}$. We will show that $M$
lies in ${\rm Im}\; h$. For this, take a conflation $X\rightarrow Q
\stackrel{d}\rightarrow X'$ with $Q$ projective. Then we have the
following commutative exact diagram
\[\xymatrix{
h_X \ar[r] \ar@{=}[d] & M \ar[r] \ar@{.>}[d] & h_Y
\ar@{.>}[d]^-{\theta}\\
h_X \ar[r] & H_Q \ar[r]^{h_d} & h_{X'}}\] Here we use that $h_Q=H_Q$
 is injective in ${\rm CM}(\mathcal{P})$; see Lemma \ref{lem:lem3.5}. From this diagram we have
 a conflation $M \rightarrow h_Y\oplus H_Q \stackrel{(\theta, h_d)}\rightarrow h_X'$
 in ${\rm CM}(\mathcal{P})$.
  By the second step there
 exists a morphism $\mu\colon Y\rightarrow X'$ such that
 $h_\mu=\theta$.
 By Lemma \ref{lem:lem1} there exists a conflation
 $Z\rightarrow Y\oplus Q\stackrel{(\mu, d)}\rightarrow X'$
 in $\mathcal{A}$ for some object $Z$. Applying $h$ to this conflation  we get an isomorphism $M\simeq h_Z$.

\vskip 3pt

 \emph{Step 4.} We will show that the functor $h$ induces an
 equivalence of exact categories between $\mathcal{A}$ and ${\rm Im}\;
 h$. Then we are done with the proof. What remains to show is that the functor $h$
 \emph{reflects exactness}, that is, any sequence $\eta\colon X\stackrel{i}\rightarrow Y \stackrel{d}\rightarrow Z$
in $\mathcal{A}$ is a conflation provided that $h_\eta$ is a
conflation in ${\rm Im}\; h$. View the functor $h$ as a full
embedding. Since $h_i={\rm Ker}\; h_d$, we have $i={\rm Ker}\; d$.
For each projective object $P$, we have an isomorphism ${\rm
Hom}_\mathcal{A}(P, \eta)\simeq {\rm Hom}_{{\rm
CM}(\mathcal{P})}(H_P, h_\eta)$ of sequences, and hence they are
both exact. In particular, for a chosen deflation
$P\stackrel{d'}\rightarrow Z$ with $P$ projective, there exists a
morphism $t\colon P\rightarrow Y$ such that $d\circ t=d'$. Now we
apply Lemma \ref{lem:lem2} to the morphism $d$. Then $d$ is a
deflation and as its kernel, $i$ is an inflation. Consequently, the
sequence $\eta$ is a conflation in $\mathcal{A}$, completing the
proof.
\end{proof}

We will apply Theorem \ref{thm:partII} to recover a  part of
\cite[Theorem C]{KLM2}, which gives a surprising link between the
category of vector bundles on weighted projective lines and a
certain submodule category.

\begin{exm}\label{exm:KLMcontinued}
Consider the factor  Frobenius category
$\mathcal{A'}={\rm vect}\; \mathbb{X}/[\mathcal{F}]$ in Example \ref{exm:KLM}.
The full subcategory $\mathcal{P}'$ consisting of projective objects
is equal to $\mathcal{P}/[\mathcal{F}]$. Then Theorem
\ref{thm:partII} implies that the associated restricted Yoneda
functor $h\colon \mathcal{A}'\rightarrow {\rm Mod}\; \mathcal{P}'$
induces an equivalence of exact categories between $\mathcal{A}'$
with an admissible subcategory of ${\rm CM}(\mathcal{P}')$. This
might be viewed as a part of \cite[Theorem C]{KLM2}.

In this situation, a highly nontrivial result is that the corresponding
admissible subcategory is ${\rm CM}(\mathcal{P}')$ itself; compare \cite[Propositon
3.18]{KLM2}. Finally observe that the category ${\rm
CM}(\mathcal{P}')$ of Cohen-Macaulay $\mathcal{P}'$-modules is equal to the
submodule category ${\mathcal{S}(\widetilde{p})}$ of Ringel-Schmidmeier
(by combining \cite[Lemma B]{KLM2} and a graded version of \cite[Lemma 4.3]{Ch09}). From
these we conclude that there is an equivalence of exact categories between 
the factor category ${\rm vect}\; \mathbb{X}/[\mathcal{F}]$ and the category 
$\mathcal{S}(\widetilde{p})$ of Ringel-Schmidmeier; 
this is  \cite[Theorem C]{KLM2}. \hfill $\square$
\end{exm}

In the next example, we apply Theorem \ref{thm:partII}  to the factor Frobenius
category obtained in Example \ref{exm:mf}.

\begin{exm}\label{exm:mfcontinued}
Let $R$ be a commutative noetherian ring and let $f\in R$ be a
regular element. We consider the factor Frobenius category $\mathcal{A}={\rm MF}_R(f)/[\mathcal{F}]$ in
Example \ref{exm:mf}.  Observe that its full subcategory $\mathcal{P}$ of projective objects
is the additive closure of the object $T:=(f{\rm Id}_R, {\rm Id}_R)$; moreover,
the endomorphism ring of $T$ (in $\mathcal{A}$) is isomorphic to the quotient ring $S:=R/(f)$.
By a version of Morita equivalence we have an equivalence ${\rm Mod}\; \mathcal{P}\simeq \mbox{Mod}\; S$
of module categories; here $\mbox{Mod}\; S$ denotes the category of $S$-modules. Furthermore, this
equivalence restricts to an equivalence ${\rm CM}(\mathcal{P})\simeq {\rm CM}(S)$. Here, ${\rm CM}(S)$ is
the category of (maximal) Cohen-Macaulay $S$-modules (\cite{Bu10, Bel3}).
Together with this equivalence we apply Theorem \ref{thm:partII} to $\mathcal{A}$.
Then the restricted Yoneda functor
$$h\colon {\rm MF}_R(f)/[\mathcal{F}] \longrightarrow {\rm CM}(S)$$
 identifies ${\rm MF}_R(f)/[\mathcal{F}]$  as an admissible subcategory of ${\rm CM}(S)$.  We will
 describe this admissible subcategory of ${\rm CM}(S)$. For this end,
 we will first give another description of the functor $h$.
 \vskip 3pt

 Consider the following functor
$${\rm Cok}\colon {\rm MF}_R(f)\longrightarrow {\rm CM}(S)$$
which sends a matrix factorization $(d_P^0, d_P^1)$ to ${\rm Cok}\;
d_P^1$ and which acts on morphisms naturally. Observe that ${\rm
Cok}\; d_P^1$ is indeed a Cohen-Macaulay $S$-module; compare \cite[Proposition 5.1]{Ei80}.
 Note that the functor ${\rm Cok}$ is exact and vanishes on $\mathcal{F}$. Then we have an induced
 functor ${\rm Cok}\colon {\rm MF}_R(f)/[\mathcal{F}]\rightarrow {\rm CM}(S)$. We claim that
 there is a natural isomorphism between $h$ and ${\rm Cok}$. In fact, to see this isomorphism,
 it suffices to note the natural isomorphisms ${\rm Hom}_\mathcal{A}(T, (d_P^0, d_P^1))\simeq {\rm Cok}\; d_P^1$ for
 all matrix factorizations $(d_P^0, d_P^1)$.

Denote by $\mathcal{B}$ the full subcategory of ${\rm CM}(S)$ consisting of modules which, when viewed as $R$-modules, have
projective dimension at most one. Observe that $\mathcal{B}$ is an extension-closed exact subcategory of
${\rm CM}(S)$. Recall that in a matrix factorization $(d_P^0, d_P^1)$ both morphisms $d_P^0$ and
$d_P^1$ are mono. It follows that the image of the functor {\rm Cok} lies in  $\mathcal{B}$.
We claim that any module in $\mathcal{B}$ lies in the image of the functor {\rm Cok}.
 To see this, for an $S$-module $M$ in $\mathcal{B}$ we take an exact sequence
 $0\rightarrow P^1\stackrel{d_P^1}\rightarrow P^0\stackrel{\pi}\rightarrow M\rightarrow 0$ such that
 $P^i$ are finitely generated projective $R$-modules, $i=0,1$. Since $f$ vanishes on $M$, then $\pi\circ f{\rm Id}_{P^0}=0$ and then $f{\rm Id}_{P^0}$ factors uniquely
 through $d_P^1$. In this way, we obtain a morphism $d_P^0\colon P^0\rightarrow P^1$ such that
 $(d_P^0, d_P^1)$ is a matrix factorization. Observe that $M\simeq {\rm Cok}\; d_P^1$. This shows the claim.
 Recall that the two functors $h$ and ${\rm Cok}$ are isomorphic. Then we conclude that
 the essential image of $h$ is $\mathcal{B}$.
 In particular, the subcategory $\mathcal{B}\subseteq {\rm CM}(S)$ is admissible.  Hence the restricted Yoneda
 functor induces an equivalence of exact categories ${\rm MF}_R(f)/[\mathcal{F}]\simeq \mathcal{B}$. One might
 compare this with   \cite[Corollary 6.3]{Ei80} and \cite[Theorem 3.9]{Or04}.

\vskip 3pt

The situation is particularly nice if we assume that the ring $R$ is \emph{regular}, that is, $R$ has finite global
dimension. In this case, the quotient ring $S$ is Gorenstein.
Observe that each Cohen-Macaulay $S$-module has projective
dimension at most one, when viewed as an $R$-module (by \cite[Lemma
18.2(i)]{Mat}), that is, $\mathcal{B}={\rm CM}(S)$. Then the restricted
Yoneda functor $h$ induces an equivalence of exact categories ${\rm
MF}_R(f)/[\mathcal{F}]\simeq {\rm CM}(S).$
\hfill  $\square$
\end{exm}

We introduce the following notion: a Frobenius category
$\mathcal{A}$ is \emph{standard} provided that the associated
restricted Yoneda functor $h\colon \mathcal{A}\rightarrow {\rm
CM}(\mathcal{P})$ is an equivalence of exact categories; this is
equivalent by Theorem \ref{thm:partII} to that the functor $h$ is
dense. For example, one can show that a Frobenius abelian category
is standard; Example
\ref{exm:KLMcontinued} claims that the factor Frobenius category
$\mathcal{A}'$ is standard; Example \ref{exm:mfcontinued} implies that for a regular ring
$R$ and a regular element $f\in R$, the factor Frobenius category
${\rm MF}_R(f)/[\mathcal{F}]$ is standard.

In general, it would be very nice to have an intrinsic criterion on
when a Frobenius category is standard.

\section{Frobenius category from exact category}

In this section we give sufficient conditions such that on an exact
category with enough projective and enough injective objects there
exists another natural exact structure, with which the given
category becomes Frobenius. We apply the result to the morphism
category of a Frobenius abelian category; it turns out that this
morphism category has a natural Frobenius exact structure. This
observation allows us to interpret the minimal monomorphism
operation in \cite{RS06} as a triangle functor, which is right
adjoint to an inclusion triangle functor.

\vskip 5pt

Let $(\mathcal{A}, \mathcal{E})$ be an exact category with enough
projective and enough injective objects. We denote by $\mathcal{P}$
and $\mathcal{I}$ the full subcategory of $\mathcal{A}$ consisting
of projective and injective objects, respectively. Note that the
exact category $\mathcal{A}$ might not be Frobenius. The aim is to
show that under certain conditions there is a new exact structure
$\mathcal{E}'$ on $\mathcal{A}$ such that $(\mathcal{A},
\mathcal{E}')$ is a Frobenius category.

Recall that a full additive subcategory $\mathcal{S}$ of
$\mathcal{A}$ is said to be \emph{contravariantly finite} provided
that each object in $\mathcal{A}$ has a right
$\mathcal{S}$-approximation. Dually one has the notion of
\emph{covariantly finite subcategory} (\cite[Section 2]{AS81}). For
two full subcategories $\mathcal{X}$ and $\mathcal{Y}$ of
$\mathcal{A}$, denote by $\mathcal{X}\vee \mathcal{Y}$ the smallest
full additive subcategory of $\mathcal{A}$ which contains
$\mathcal{X}$ and $\mathcal{Y}$ and is closed under taking direct
summands.

Recall that for a full additive subcategory $\mathcal{S}$ of an exact category $\mathcal{A}$,
 a conflation $\eta\colon X\rightarrow Y\rightarrow Z$ is \emph{right
$\mathcal{S}$-acyclic} provided that the sequences ${\rm
Hom}_\mathcal{A}(S, \eta)$ are short exact for all $S\in
\mathcal{S}$. Dually one has the notion of \emph{left
$\mathcal{S}$-acyclic conflation}.

\vskip 5pt

 Here is our third result, which gives sufficient
conditions such that there is a natural (and new) exact structure on
$\mathcal{A}$,
 with which $\mathcal{A}$ becomes a Frobenius category.

\begin{thm}\label{thm:partIII}
Use the notation as above. Assume that $\mathcal{P}'\subseteq
\mathcal{P}$ and $\mathcal{I}'\subseteq \mathcal{I}$ are two full
additive subcategories subject to the following conditions:
\begin{enumerate}
\item $\mathcal{P}'\vee \mathcal{I}=\mathcal{I}'\vee \mathcal{P}$;
\item $\mathcal{P}'\subseteq \mathcal{A}$ is covariantly finite and
$\mathcal{I}'\subseteq \mathcal{A}$ is contravariantly finite;
\item the class of right $\mathcal{I}'$-acyclic conflations
coincides with the class of left $\mathcal{P}'$-acyclic conflations.
\end{enumerate}
Denote the class of conflations in (3) by $\mathcal{E}'$. Then the
pair $(\mathcal{A}, \mathcal{E}')$ is a Frobenius exact category.
\end{thm}

The proof of this result is quite direct, once we notice the
following general observation.

\begin{lem}\label{lem:lem3}
Let $(\mathcal{A}, \mathcal{E})$ be an exact category. For a full
additive  subcategory $\mathcal{S}\subseteq \mathcal{A}$, denote by
$\mathcal{E}'$ the class of right $\mathcal{S}$-acyclic conflations.
Then the pair $(\mathcal{A}, \mathcal{E}')$ is an exact category.
\end{lem}

\begin{proof}
For a conflation $(i, d)$ in $\mathcal{E}'$, we will temporarily
call  $i$  an $\mathcal{E}'$-inflation and $d$  an
$\mathcal{E}'$-deflation. We verify the axioms for the pair
$(\mathcal{A}, \mathcal{E}')$. Recall that the axiom (Ex1)$^{\rm
op}$ can be deduced from the others; see \cite[Appendix A]{Ke3}. So
we only show the remaining four axioms. The axiom (Ex0) is clear.
Recall that a deflation $d\colon Y\rightarrow Z$ is an
$\mathcal{E}'$-deflation if and only if every morphism from an
object in $\mathcal{S}$ to $Z$ factors through $d$. This observation
yields (Ex1) immediately.

Consider the pullback diagram in the axiom (Ex2); see Section 2.
Assume that $d\colon Y\rightarrow Z$ is an $\mathcal{E}'$-deflation.
We will show that $d'\colon Y'\rightarrow Z'$ is also an
$\mathcal{E}'$-deflation. Take a morphism $s\colon S\rightarrow Z'$
with $S\in \mathcal{S}$. Since $d$ is an $\mathcal{E}'$-deflation,
the morphism $f\circ s$ lifts to $Y$, that is, there exists
$s'\colon S\rightarrow Y$ such that $d\circ s'=f\circ s$. Using the
universal property of the pullback diagram, we infer that  there
exists a unique morphism $t\colon S\rightarrow Y'$ such that $d'
\circ t=s$ and $f' \circ t=s'$. In particular, the morphism $s$
factors through $d'$, proving that $d'$ is an
$\mathcal{E}'$-deflation.

It remains to verify (Ex2)$^{\rm op}$. Consider the pushout diagram
in (Ex2)$^{\rm op}$; see Section 2. We assume that $i\colon
X\rightarrow Y$ is an $\mathcal{E}'$-inflation.  We will show that
$i'$ is  an $\mathcal{E}'$-inflation. We apply the dual of Lemma
\ref{lem:lem1} to get the following commutative diagram such that
the two rows are conflations
\[\xymatrix{  X \ar[r]^-{i} \ar[d]^-{f} & Y \ar[d]^-{f'} \ar[r]^-d & Z \ar@{=}[d] \\
   X'  \ar[r]^-{i'} & Y' \ar[r]^-{d'} & Z }\]
Here $d=d'\circ f'$. The fact that $i$ is an
$\mathcal{E}'$-inflation implies that $d$ is an
$\mathcal{E}'$-deflation. Consider any morphism $s\colon
S\rightarrow Z$ with $S\in \mathcal{S}$. Then $s$ factors through
$d$. Since $d=d'\circ f'$, we infer that the morphism $s$ factors
through $d'$. This shows that $d'\colon Y'\rightarrow Z$ is an
$\mathcal{E}'$-deflations and then $i'$ is an
$\mathcal{E}'$-inflation. We are done.
\end{proof}

\vskip 10pt

\noindent {\bf Proof of Theorem \ref{thm:partIII}:}\quad By Lemma
\ref{lem:lem3} the pair $(\mathcal{A}, \mathcal{E}')$ is an exact
category. We will call a conflation in $\mathcal{E}'$ an
$\mathcal{E}'$-conflation.  Note  by the condition (3) that the
objects in $\mathcal{P}'\vee \mathcal{I}=\mathcal{I}'\vee
\mathcal{P}$ are projective and injective in the exact category
$(\mathcal{A}, \mathcal{E}')$.

Observe from the condition (3)  that a conflation
$X\stackrel{i}\rightarrow Y \stackrel{d}\rightarrow Z$ is an
$\mathcal{E}'$-conflation if and only if any morphism from an object
in $\mathcal{I}'$ to $Z$ factors through $d$, if and only if any
morphism from $X$ to an object in $\mathcal{P}'$ factors through
$i$. For an object $Z$ in $\mathcal{A}$, take a conflation $d\colon
P\rightarrow Z$ with $P\in \mathcal{P}$ and a right
$\mathcal{I}'$-approximation $s\colon I'\rightarrow Z$. By Lemma
\ref{lem:lem1} the morphism $(d, s)\colon P\oplus I'\rightarrow Z$
is an deflation. It induces a conflation $\eta\colon X\rightarrow
 P\oplus I'\stackrel{(d, s)}\rightarrow Z $.
 From the observation just made, we
obtain that the conflation $\eta$ is an $\mathcal{E}'$-conflation.
This proves that the exact category $(\mathcal{A}, \mathcal{E}')$
has enough projective objects and the class of projective objects
is equal to $\mathcal{I}'\vee \mathcal{P}$. Dually one shows that
the exact category $(\mathcal{A}, \mathcal{E}')$ has enough
injective objects and the class of injective objects is equal to
$\mathcal{P}'\vee \mathcal{I}$. Then we conclude that the exact
category $(\mathcal{A}, \mathcal{E}')$ is Frobenius, completing the
proof. \hfill $\square$

\vskip 5pt

We apply Theorem \ref{thm:partIII} to the morphism category of a
Frobenius abelian category. This allows us to interpret the minimal
monomorphism operation (\cite{RS06}) as a right adjoint to an
inclusion triangle functor.

\begin{exm}
Let $\mathcal{A}$ be an abelian category. Denote by ${\rm
Mor}(\mathcal{A})$ the \emph{morphism category} of $\mathcal{A}$:
its objects are given by morphisms $\alpha\colon X\rightarrow Y$ in
$\mathcal{A}$, and morphisms $(f, g)\colon \alpha \rightarrow
\alpha'$ are given by commutative squares in $\mathcal{A}$, that is,
two morphisms $f\colon X\rightarrow X'$ and $g\colon Y\rightarrow
Y'$ such that $\alpha'\circ f=g\circ \alpha$. It is an abelian
category; a sequence $\alpha \rightarrow \alpha'\rightarrow
\alpha''$ in ${\rm Mor}(\mathcal{A})$ is exact if and only if the
corresponding sequences of domains and targets are exact in
$\mathcal{A}$; see \cite[Corollary 1.2]{FGR75}.

Assume that the abelian category $\mathcal{A}$ is Frobenius. In
general the abelian category ${\rm Mor}(\mathcal{A})$ is not
Frobenius. In fact, the  category ${\rm Mor}(\mathcal{A})$
 has enough projective and injective objects; projective objects are
 equal to objects of the form
 $(0\rightarrow P)\oplus (Q\stackrel{\rm Id_Q}\rightarrow
 Q)$ for some projective objects $P, Q\in \mathcal{A}$; dually injective
 objects are equal to objects of the form $(P\rightarrow 0)\oplus
  (Q\stackrel{\rm Id_Q}\rightarrow Q)$ for some injective objects $P, Q\in \mathcal{A}$;
  compare \cite[Section 2]{RS06}.
  Denote by $\mathcal{P}$ and
  $\mathcal{I}$ the full subcategory  consisting of projective and injective objects in
  ${\rm Mor}(\mathcal{A})$,
  respectively.

  Take $\mathcal{P}'\subseteq \mathcal{P}$ to be the
  full subcategory consisting of objects of the form $0\rightarrow
  P$. Take $\mathcal{I}'\subseteq \mathcal{I}$ to be the full
  subcategory consisting of objects of the form $P\rightarrow 0$. We
  will verify the conditions in Theorem \ref{thm:partIII}.

  The condition (1) is clear. To see (2), take  an object $\alpha \colon X\rightarrow Y$
  in ${\rm Mor}(\mathcal{A})$ and consider its cokernel $\pi\colon Y\rightarrow {\rm Cok}\; \alpha$
  and a monomorphism $i\colon {\rm Cok}\; \alpha\rightarrow P$ with
  $P$ injective. Then the morphism $(0, i\circ \pi)\colon \alpha \rightarrow (0\rightarrow
  P)$ is a left $\mathcal{P}'$-approximation. This proves that
  $\mathcal{P}'\subseteq {\rm Mor}(\mathcal{A})$ is covariantly finite. Dually
  $\mathcal{I}'\subseteq {\rm Mor}(\mathcal{A})$ is
  contravariantly finite. For (3), observe that a short exact
  sequence $ 0\rightarrow \alpha \rightarrow
\alpha'\rightarrow \alpha'' \rightarrow 0$ in ${\rm
Mor}(\mathcal{A})$ is left $\mathcal{P}'$-acyclic if and only if the
corresponding sequence of cokernels is exact; by Snake Lemma this is
equivalent to that the corresponding sequence of kernels is exact,
and then equivalent to that the sequence is right
$\mathcal{I}'$-acyclic.

\vskip 3pt

 We apply Theorem \ref{thm:partIII} to obtain a Frobenius
exact structure on ${\rm Mor}(\mathcal{A})$. Note that the
corresponding conflations are given by short exact sequences in ${\rm
Mor}(\mathcal{A})$ such that the associated sequences of kernels
and cokernels are exact in $\mathcal{A}$; moreover, projective
objects are equal to objects of the form $(0\rightarrow P)\oplus
(Q\stackrel{\rm Id_Q}\rightarrow Q)\oplus (R\rightarrow 0)$ for some
projective objects $P, Q, R\in \mathcal{A}$. Denote by
$\mathcal{P}_{\rm new}$ the full subcategory of ${\rm Mor}(\mathcal{A})$
formed by these objects. We denote by $\underline{\rm Mor}(\mathcal{A})$ the stable
category of  ${\rm Mor}(\mathcal{A})$ modulo $\mathcal{P}_{\rm
new}$; it is a triangulated category (\cite{Ha1, Ke3}).

Recall that ${\rm Mon}(\mathcal{A})$ is the  extension-closed exact
subcategory of ${\rm Mor}(\mathcal{A})$ consisting of monomorphisms
in $\mathcal{A}$; it is called the \emph{monomorphism category} of
$\mathcal{A}$. In fact, it is a Frobenius category such that its
projective objects are equal to  objects of the form $(0\rightarrow
P)\oplus (Q\stackrel{{\rm Id}_Q}\rightarrow Q)$ for projective
objects $P, Q\in \mathcal{A}$. Denote by $\underline {\rm
Mon}(\mathcal{A})$ the stable category. For details,  see
\cite{Ch09}. Hence we have an inclusion triangle functor ${\rm
inc}\colon \underline{\rm Mon}(\mathcal{A}) \hookrightarrow
\underline{\rm Mor}(\mathcal{A})$. It is remarkable that this
functor admits a right adjoint ${\rm inc}_\rho\colon \underline{\rm
Mor}(\mathcal{A})\rightarrow \underline{\rm Mon}(\mathcal{A})$: for
each object $\alpha\colon X\rightarrow Y$, consider a monomorphism
$i_X\rightarrow I(X)$ such that $I(X)$ injective, and set ${\rm
inc}_\rho(\alpha)=\binom{\alpha}{i_X}\colon X\rightarrow Y\oplus
I(X)$; the action of ${\rm inc}_\rho$ on morphisms is defined
naturally. In particular, the functor ${\rm inc}_\rho$ is a triangle
functor; see \cite[Section 8]{Ke}.

 Suppose that the abelian category $\mathcal{A}$ has injective hulls. For an object
 $\alpha\colon X\rightarrow Y$ in ${\rm Mor}(\mathcal{A})$, consider
 its kernel $i\colon K\rightarrow X$ and an injective hull $j\colon K\rightarrow
 I(K)$. Then  there exists a morphism
 $\bar{i}\colon X\rightarrow I(K)$ such that $\bar{i}\circ i=j$.
 Note that $\binom{\alpha}{\bar{i}}\colon X\rightarrow Y\oplus I(K)$
 is a monomorphism; it is called the \emph{minimal
monomorphism} associated to $\alpha$ (\cite[Sections 2,4]{RS06}).
Denote the minimal monomorphism by ${\rm Mimo}(\alpha)$. It is
remarkable that there is a natural isomorphism between the object
${\rm inc}_\rho(\alpha)$ and  ${\rm Mimo}(\alpha)$ in the stable
category $\underline{\rm Mon}(\mathcal{A})$; compare \cite[Section
4, Claim 2]{RS06}. Then the minimal monomorphism operation ${\rm
Mimo}(\mbox{-})$ becomes naturally  a triangle functor. \hfill
$\square$
\end{exm}

We would like to point out  that  the Frobenius category ${\rm
Mor}(\mathcal{A})$ in the above example is standard, that is, the
associated restricted Yoneda functor yields an equivalence ${\rm
Mor}(\mathcal{A})\simeq {\rm CM}(\mathcal{P}_{\rm new})$ of exact
categories; see Section 4. Indeed, both exact categories are
equivalent to the category of left exact sequences in $\mathcal{A}$
with the obvious exact structure. This observation and its
generalization will be treated elsewhere.

\bibliography{}

\begin{thebibliography}{999}

\bibitem{ABr69} {\sc M. Auslander and M. Bridger}, {\em Stable module
category}, Mem. Amer. Math. Soc. {\bf 94}, 1969.

\bibitem{AR91} {\sc M. Auslander and I. Reiten}, {\em Applications of contravariantly
finite subcategories}, Adv. Math. {\bf 86} (1991), 111--152.

\bibitem{ARS}{\sc M. Auslander, I. Reiten, and  S.O. Smal{\o}},
Representation Theory of Artin Algebras. Cambridge Studies in Adv.
Math. {\bf 36}, Cambridge Univ. Press,  Cambridge, 1995.


\bibitem{AS81} {\sc M. Auslander and S.O. Smal{\o}}, {\em Almost split sequencrs
in subcategories}, J. Algebra {\bf 69} (1981), 426--454.


\bibitem{AM}  {\sc L.L. Avramov and A. Martsinkovsky,} {\em Absolute, relative and Tate
cohomology of modules of finite Gorenstein dimension}, Proc. London
Math. Soc. (3) {\bf 85} (2002), 393--440.


\bibitem{Bel3} {\sc A. Beligiannis,} {\em Cohen-Macaulay modules,
(co)torsion pairs and virtually Gorenstein algebras,} J. Algebra
{\bf 288} (2005), 137-211.

\bibitem{Buc87} {\sc R.O. Buchweitz,} Maximal Cohen-Macaulay Modules
and Tate Cohomology over Gorenstein Rings. Unpublished manuscript,
1987.

\bibitem{Bu10} {\sc T. B\"{u}hler,} {\em Exact categories}, Expo.
Math. {\bf 28} (2010), 1--69.



\bibitem{Ch10} {\sc X.W. Chen}, {\em Relative singularity categories and Gorenstein
projective modules}, Math. Nath., to appear; arXiv:0709.1762v2.

\bibitem{Ch09} {\sc X.W. Chen}, {\em The stable monomorphism category of a Frobenius
category}, arXiv:0911.1987v2.


\bibitem{CFH06} {\sc L.W. Christensen, A. Frankild and H. Holm}, {\em On
Gorenstein projective, injective and flat dimensions-- a functorial
description with applications}, J. Algebra {\bf 302} (2006),
231--279.

\bibitem{Ei80} {\sc D. Eisenbud,} {\em Homological algebra on a complete intersection, with
an application to group representations}, Trans. Amer. Math. Soc. {\bf 260} (1) (1980), 35--64.


\bibitem{EJ} {\sc E.E. Enochs, O.M.G. Jenda}, Relative Homological Algebra. De
Gruyter Expositions in Math. {\bf 30}, Walter de Gruyter, Berlin,
New York, 2000.


\bibitem{FGR75} {\sc R.M. Fossum, P. Griffith  and I. Reiten}, Trivial Extensions of Abelian Categories. Lecture
Notes in Math. {\bf 456}, Springer-Verlag, Berlin Heidelberg New
York, 1975.


\bibitem{GL} {\sc W. Geigel and H. Lenzing,} {\em A class of weighted projective curves arising in
representation theory of finite dimensional algebras,} in:
Singularities, representations of algebras and vector bundles,
Lecture Notes in Math. {\bf 1273}, 265--297, Springer, 1987.


\bibitem{Ha1}{\sc D. Happel,} Triangulated Categories in the
Representation Theory of Finite Dimensional Algebras.  London Math.
Soc. Lecture Notes Ser. {\bf 119}, Cambridge Univ. Press, Cambridge, 1988.

\bibitem{He60}{\sc A. Heller,} {\em The loop-space functor in homological
algebra}, Trans. Amer. Math. Soc. {\bf 96} (1960), 382--394.

\bibitem{Ke3} {\sc B. Keller,} {\em Chain complexes and stable
categories,} Manuscripta Math. {\bf 67} (1990), 379-417.

\bibitem{Ke} {\sc B. Keller,} {\em Derived categories and their
uses,} Handbook of Algebra {\bf 1}, 671--701, North-Holland,
Amsterdam, 1996.


\bibitem{KLM1} {\sc D. Kussin, H. Lenzing and H. Meltzer}, {\em Triangle singularities, ADE-chains
and weighted projective lines}, preprint.

\bibitem{KLM2} {\sc D. Kussin, H. Lenzing and H. Meltzer}, {\em Nilpotent operators and
weighted projective lines}, arXiv:1002.3797v1.

\bibitem{Mat} {\sc H. Matsumura}, Commuative Ring Theory. Cambridge Studies in Advanced Math. {\bf 8},
Cambridge Univ. Press, Cambridge, 1986.

\bibitem{Mit72} {\sc B. Mitchell}, {\em Rings with several objects},
Adv. Math. {\bf 8} (1972), 1--161.

\bibitem{Or04} {\sc D. Orlov}, {\em Triangulated categories of singularities and D-branes in Landau-Ginzburg
models}, Trudy Steklov Math. Institute {\bf 204} (2004), 240--262.


\bibitem{Qui73} {\sc D. Quillen,} {\em Higher algebraical K-theory
I}, Springer Lecture Notes in Math. {\bf 341}, 1973, 85--147.

\bibitem{RS06}{\sc C.M. Ringel and M. Schmidmeier,} {\em The Auslander-Reiten translation in
submodule category}, Trans. Amer. Math. Soc. {\bf 360} (2) (2008),
691--716.


\bibitem{RS08} {\sc C.M. Ringel and M. Schmidmeier,} {\em Invariant subspaces of
nilpotent linear operators}, J. Reine Angew. Math. {\bf 614} (2008),
1--52.



\end{thebibliography}

\vskip 10pt

 {\footnotesize \noindent Xiao-Wu Chen, Department of
Mathematics, University of Science and Technology of
China, Hefei 230026, P. R. China \\
Homepage: http://mail.ustc.edu.cn/$^\sim$xwchen \\
\emph{Current address}: Institut fuer Mathematik, Universitaet
Paderborn, 33095, Paderborn, Germany}

\end{document}